\documentclass[reqno,a4paper]{amsart}
\usepackage{amsmath}
\usepackage{amssymb,amsfonts}
\usepackage{amsthm}
\usepackage{enumerate}
\usepackage[mathscr]{eucal}
\usepackage{eqlist}

\theoremstyle{plain}
\newtheorem{theorem}{Theorem}[section]
\newtheorem{prop}[theorem]{Proposition}

\newtheorem{corol}{Corollary}[theorem]

\theoremstyle{definition}
\newtheorem{definition}{Definition}[section]
\newtheorem{remark}{\textnormal{\textbf{Remark}}}

\theoremstyle{remark}

\numberwithin{equation}{section}

\begin{document}

\title[Elliptic Curves Containing Sequences of Consecutive Cubes]
{Elliptic Curves Containing Sequences of Consecutive Cubes }

\author{ Gamze Sava\c{s} \c{C}EL\.{I}K and G\"{o}khan Soydan}

\address{{\bf Gamze Sava\c{s} \c{C}elik}\\
Department of Mathematics \\
Uluda\u{g} University\\
 16059 Bursa, Turkey}
\email{gamzesavas91@gmail.com; gamzesavascelik@gmail.com }

\address{{\bf G\"{o}khan Soydan} \\
	Department of Mathematics \\
	Uluda\u{g} University\\
	16059 Bursa, Turkey}
\email{gsoydan@uludag.edu.tr }

\newcommand{\acr}{\newline\indent}

\thanks{}

\subjclass[2010]{Primary 14G05, Secondary 11B83}
\keywords{Elliptic curves, rational points, sequences of consecutive cubes.}

\begin{abstract}
Let $E$ be an elliptic curve over $\mathbb{Q}$ described by $y^2= x^3+ Kx+ L$ where $K, L \in \mathbb{Q}$. A set of rational points $(x_i,y_i) \in E(\mathbb{Q})$ for $i=1, 2, \cdots, k$, is said to be a sequence of consecutive cubes on $E$ if the $x-$coordinates of the points $x_i$'s for $i=1, 2, \cdots$ form consecutive cubes. In this note, we show the existence of an infinite family of elliptic curves containing a length-$5$-term sequence of consecutive cubes. Morever, these five rational points in $E (\mathbb{Q})$ are linearly independent and the rank $r$ of $E(\mathbb{Q})$ is at least $5$.
\end{abstract}

%\dedicatory{Dedicated to K\'{a}lm\'{a}n Gy\H{o}ry on the occasion of his 76th birthday}
\maketitle

%################################%
\section{Introduction}\label{sec:1}
%################################%

Let us consider a rational elliptic curve given by a Weierstrass equation
\begin{equation} \label{eq:1}
y^2+a_1xy+a_3y=x^3+a_2x^2+a_4x+a_6
\end{equation}
with $a_1, \cdots, a_6 \in \mathbb{Q}$. We will say that the points $(x_i, y_i), i=1, \cdots, k$ on the curve   \eqref{eq:1} are in arithmetic progression of length $k$ if the sequence $x_1, x_2, \cdots, x_k$ forms an arithmetic progression (AP for short).\\

In 1992, Lee and V\'{e}lez, \cite{LV}, found infinitely many curves of type $y^{2}=x^{3}+a$ containing $k=4$-length APs. In 1999, Bremner, \cite{Br}, showed that there are infinitely many elliptic curves with $k=7$ and $k=8$-length APs. We shall briefly say $k$-AP instead of $k$-length AP. Four years later, Campbell, \cite{Cam}, gave a different method to produce infinite families of elliptic curves with $k=7$ and $k=8$ APs. In addition, he described a method for obtaining infinite families of quartic elliptic curves with $k=9$ AP and gave an example of a quartic elliptic curve with $k=12$ AP. Two years later, Ulas \cite{Ul}, first described a construction method for an infinite family of quartic elliptic curves on which there exists an AP with $k=10$. Secondly he showed that there is an infinite family of quartics containing AP with $k=12$. In $2006$, Macleod, \cite{Mac}, showed that by simplifying Ulas' approach, more general parametric solutions for APs arise with $k=10$ which give a large number of examples with $k=12$ and a few with $k=14$.\\

Let $f(x)$ be an irreducible polynomial over $\mathbb{Q}$ of degree five. Consider the hyperelliptic curve $y^2=f(x)$. In $2009$, Ulas, \cite{U2}, found an infinite family of curves on which there is an AP  with $k=11$. In the same year, Alvarado, \cite{Al}, showed the existence of an infinite family of curves which contain APs with $k=12$. Recently Dey and Maji, \cite{DM}, found upper bounds for the lengths of sequences of rational points on Mordell curves which are defined by the equations of the type $y^2=x^3+k$, $k \in \mathbb{Q} \backslash \{0\}$, such that the ordinates of the points are in AP, and also when both the abscissae and ordinates of the points are seperately the terms of two APs.\\

In $2013$, Bremner and Ulas, \cite{BU}, considered the sequences of rational points on elliptic curves whose $x-$coordinates form a \textquotedblleft geometric progression\textquotedblright in $\mathbb{Q}$. They obtained an infinite family of elliptic curves having geometric progression sequence of length $4$ and they also pointed out infinitely many elliptic curves with length $5$ geometric progression sequences can be obtained.\\

Recently, Kamel and Sadek, \cite{KS}, considered sequences of rational points on elliptic curves given by the equation $y^2=ax^3+bx+c$ over $\mathbb{Q}$ whose $x-$coordinates form a sequence of consecutive squares. They showed that elliptic curves given by the latter equation with $5-$term sequences of rational points whose $x-$coordinates are elements of a sequence of consecutive squares in $\mathbb{Q}$ parametrized by an elliptic surface whose rank is positive. This implies the existence of infinitely many such elliptic curves. They also showed that these five rational points in the sequence are linearly independent in the group of rational points of the elliptic curve they lie on. Especially, they introduced an infinite family of elliptic curves of rank$\geq 5$.\\

In this work, we investigate sequences of rational points on elliptic curves whose $x-$coordinates form a sequence of \textquotedblleft consecutive cubes \textquotedblright. We consider elliptic curves given by the equation  $y^2=kx^3+lx+m$ over $\mathbb{Q}$. Following the strategy in \cite{KS}, we obtain all their results which we detailed in the previous paragraph for \textquotedblleft consecutive cubes\textquotedblright.

%################################%
\section{SEQUENCES OF CONSECUTIVE CUBES}\label{sec:2}
%################################%

\begin{definition}\label{def:2.1}
Let $E$ be an elliptic curve defined over a number field $F$ by the Weierstrass equation
\begin{equation}  \label{eq:2}
y^2+a_1xy+a_3y=x^3+a_2x^2+a_4x+a_6, \ a_i \in F  .
\end{equation}
The points $(x_i, y_i) \in E(F)$ are said to form a sequence of \textquotedblleft consecutive cubes\textquotedblright on $E$ if there is  $c \in F$ such that $x_i=(c+i)^3, i=1, 2, \cdots$.
\end{definition} 

Now we first need a result which guarantees the finiteness of the sequence of consecutive cubes on an elliptic curve. In $1910$, Mordell conjectured that if $\varepsilon$  is an algebraic curve over $F$ of genus $g \geq 2$, then there are only finitely many rational points on $\varepsilon$, i.e. the set $\varepsilon (F)$ of $F-$rational points is finite. In $1983$, this conjecture was proved by Faltings, \cite{Fa}, and hence is now also known as Faltings' theorem. So, using Faltings' theorem we give the following proposition about the finiteness of consecutive cubes on an elliptic curve.

\begin{prop}\label{prop:2.2}
Let $E$ be an elliptic curve defined by \eqref{eq:2} over a number field $F$. Let $(x_i, y_i) \in E(F)$ be a sequence of consecutive cubes on $E$. Then the sequence $(x_i, y_i)$ is finite.
\end{prop}

\begin{proof}
Assume without loss of generality that $x_i=(c+i)^3,\ i=1, 2, \cdots, c \in F$. This sequence leads to a sequence of rational points on the genus $5$ hyperelliptic curve
\begin{equation*}
E': y^2+a_1x^3y+a_3y=x^9+a_2x^6+a_4x^3+a_6.
\end{equation*}

Thus, the points $(c+i, y) \in E'(F)$. By using Faltings' theorem \cite{Fa}, we obtain that $E'(F)$ is finite, so the sequence is finite.
\end{proof}

Next, using the above proposition, we define the length of this sequence like AP.

\begin{definition}\label{def:2.3}
Let $E$ be an elliptic curve over $\mathbb{Q}$ defined by a Weierstrass equation. Let $(x_i, y_i) \in E(\mathbb{Q}),\ i=1, 2, \cdots, n$, be a sequence of consecutive cubes on $E$. Then $n$ is called the length of the sequence.
\end{definition}

%################################%
\section{CONSTRUCTING ELLIPTIC CURVES CONTAINING 5- TERM SEQUENCES OF CONSECUTIVE CUBES }\label{sec:3}
%################################%

In this section, we investigate a family of  elliptic curves given by the affine equation

\begin{equation} \label{eq:3}
E: y^2=kx^3+lx+m
\end{equation} over $\mathbb{Q}.$ We will show that there exist infinitely many elliptic curves given by the latter equation having $5-$term sequences of consecutive cubes.\\

We consider $3-$term sequences of consecutive cubes. So, if $((c-1)^3,p), (c^3,q)$, and $((c+1)^3,r)$ lie in $E(\mathbb{Q})$, where $c \in \mathbb{Q}$, then these rational points form a $3-$term sequence of consecutive cubes. Using these points, we obtain 
\begin{equation*}
p^2=k(c-1)^9+l(c-1)^3+m,
\end{equation*}
\begin{equation*}
q^2=kc^9+lc^3+m,
\end{equation*}
\begin{equation*}
r^2=k(c+1)^9+l(c+1)^3+m.
\end{equation*}

Hence, solving this system gives the following

\begin{equation*}
k=[(3c^2+3c+1)p^2+(-6c^2-2)q^2+(3c^2-3c+1)r^2]/6c(27c^8+54c^6+c^2+2),
\end{equation*}
\begin{equation*}
\begin{aligned}
 l &=[-(9c^8+36c^7+84c^6+126c^5+126c^4+84c^3+36c^2+9c+1)p^2+(18c^8\\
 &+168c^6+252c^4+72c^2+2)q^2-(9c^8-36c^7+84c^6-126c^5+126c^4-84c^3\\
 &+36c^2-9c+1)r^2]/6(3c^2-3c+1)(9c^6+9c^5+24c^4+21c^3+13c^2+6c+2)c,
\end{aligned}
\end{equation*}
\begin{equation}\label{eq:4}
\begin{aligned}
m&=[(6c^{10}+33c^9+83c^8+126c^7+126c^6+84c^5+36c^4+9c^3+c^2)p^2\\
&+(-12c^{10}-4c^8+72c^6-72c^4+4c^2+12)q^2+(6c^{10}-33c^9+83c^8-126c^7\\
&+126c^6-84c^5+36c^4-9c^3+c^2)r^2]/6(3c^2-3c+1)(9c^6+9c^5+24c^4+21c^3\\
&+13c^2+6c+2).
\end{aligned}
\end{equation}

Then, we obtain the following result:

\begin{remark}\label{rem:3.1}
As above, for given $p, q, r \in \mathbb{Q}(c)$, we know the existence of $k, l, m \in \mathbb{Q}(c)$ such that the ordered pairs $((c-1)^3, p), (c^3,q)$ and $((c+1)^3,r)$ are three rational points lying on \eqref{eq:3}.
\end{remark}

Now, secondly assuming $((c+2)^3,s)$ is a rational point on \eqref{eq:3}, we obtain a $4-$term sequence of consecutive cubes on \eqref{eq:3}. Putting the values $k, l, m$ and $((c+2)^3,s)$ in \eqref{eq:3}, one can find
\begin{equation}\label{eq:5}
\begin{aligned}
s^2&=[(84+2109c^2+626c+243c^8+27c^9+1026c^7+2646c^6+4536c^5+5292c^4\\
&+4159c^3)p^2+(-1674c^7-1950c^2-762c-3951c^3-5544c^4-486c^8\\
&-3780c^6-5544c^5-168-81c^9)q^2+(702c^7+1134c^6+138c-159c^2+243c^8\\
&+81c^9+252c^4+1008c^5+84-207c^3)r^2]/(3c^2-3c+1)(3c^2+1)(1+3c^2\\
&+3c)(c^2+2)c.
\end{aligned}
\end{equation}

Thus, we need to find the elements $p, q, r$ and $s$ in $\mathbb{Q}(c)$ which satisfy the equation \eqref{eq:5}. Now let's explain how to find the general solution $(p, q, r, s)$ for equation \eqref{eq:5}.\\

Consider the quadratic surface
\begin{equation*}
S: a_1x^2+a_2y^2+a_3z^2+a_4t^2=0
\end{equation*}
over $\mathbb{Q}$ and the line 
\begin{equation*}
aQ_1+bQ_2=(a+bu_1: a+bv_1: a+bw_1: a)
\end{equation*}
connecting the rational points $Q_1=(1:1:1:1)$ and $Q_2=(u_1: v_1: w_1: 0)$ lying on $S$ in three-dimensional projective space $\mathbb{P}^3$. The intersection of $S$ and $aQ_1+bQ_2$ yields the quadratic equation
\begin{equation*}
(a_1+a_2+a_3+a_4)a^2+(a_1u_1^2+a_2v_1^2+a_3w_1^2)b^2+(2a_1u_1+2a_2v_1+2a_3w_1)ab=0.
\end{equation*}
Using $Q_1$ and $Q_2$ lying on $S$, one solves this quadratic equation and obtains formulae for solutions $(x, y, z, t)$. Since $(p, q, r, s)=(1, 1, 1, 1)$ is a solution for equation \eqref{eq:5}, applying the above procedure to \eqref{eq:5} gives the general solution $(p, q, r, s)$ with the following parametrization:

\begin{equation*}
\begin{aligned}
p&=\left( c+1 \right)  \left( 3\,{c}^{2}+3\,c+1
\right)  \left( 3\,{c}^{2}+6\,c+4 \right)  \left( 3\,{c}^{2}+9\,c+7
\right)  \left( {c}^{2}+2\,c+3 \right) {u}^{2}\\
&+3\, \left( {c}^{2}+c+1
\right)  \left( 3\,{c}^{2}+1 \right)  \left( 3\,{c}^{2}+9\,c+7
\right)  \left( 3\,{c}^{3}+6\,{c}^{2}+18\,c+8 \right) {v}^{2}\\
&-3\,\left( {c}^{2}+c+1 \right)  \left( 3\,{c}^{2}-3\,c+1 \right)  \left( 
3\,{c}^{2}+6\,c+4 \right)  \left( 3\,{c}^{3}+3\,{c}^{2}+15\,c+7
\right) {w}^{2}\\
&-6\, \left( {c}^{2}+c+1 \right)  \left( 3\,{c}^{2}+1 \right)  \left( 3
\,{c}^{2}+9\,c+7 \right)  \left( 3\,{c}^{3}+6\,{c}^{2}+18\,c+8
\right) vu\\
&+6\, \left( {c}^{2}+c+1 \right)  \left( 3\,{c}^{2}-3\,c+1
\right)  \left( 3\,{c}^{2}+6\,c+4 \right)  \left( 3\,{c}^{3}+3\,{c}^{
2}+15\,c+7 \right) wu,
\end{aligned}
\end{equation*}
\begin{equation*}
\begin{aligned}
q&=- \left( c+1 \right)  \left(3\,{c}^{2}+3
\,c+1 \right)  \left( 3\,{c}^{2}+6\,c+4 \right)  \left( 3\,{c}^{2}+9\,c+
7 \right)  \left( {c}^{2}+2\,c+3 \right) {u}^{2}\\
&-3\, \left( {c}^{2}+c+1 \right)  \left( 3\,{c}^{2}+1 \right)  \left( 3\,{c}^{2}+9\,c+7
\right)  \left( 3\,{c}^{3}+6\,{c}^{2}+18\,c+8 \right) {v}^{2}\\
&-3\,\left( {c}^{2}+c+1 \right)  \left( 3\,{c}^{2}-3\,c+1 \right)  \left( 
3\,{c}^{2}+6\,c+4 \right)  \left( 3\,{c}^{3}+3\,{c}^{2}+15\,c+7
\right) {w}^{2}\\
&2\, \left( c+1 \right)  \left( 3\,{c}^{2}+3\,c+1 \right)  \left( 3\,{c
}^{2}+6\,c+4 \right)  \left( 3\,{c}^{2}+9\,c+7 \right)  \left( {c}^{2}
+2\,c+3 \right) vu\\
&+6\, \left( {c}^{2}+c+1 \right)  \left( 3\,{c}^{2}-3
\,c+1 \right)  \left( 3\,{c}^{2}+6\,c+4 \right)  \left( 3\,{c}^{3}+3\,
{c}^{2}+15\,c+7 \right) wv,
\end{aligned}
\end{equation*}
\begin{equation*}
\begin{aligned}
r&=- \left( c+1 \right)  \left( 3\,{c}^{2}+3\,c+1
\right)  \left( 3\,{c}^{2}+6\,c+4 \right)  \left( 3\,{c}^{2}+9\,c+7
\right)  \left( {c}^{2}+2\,c+3 \right) {u}^{2}\\
&+3\, \left( {c}^{2}+c+1
\right)  \left( 3\,{c}^{2}+1 \right)  \left( 3\,{c}^{2}+9\,c+7
\right)  \left( 3\,{c}^{3}+6\,{c}^{2}+18\,c+8 \right) {v}^{2}\\
&+3\,
\left( {c}^{2}+c+1 \right)  \left( 3\,{c}^{2}-3\,c+1 \right)  \left( 
3\,{c}^{2}+6\,c+4 \right)  \left( 3\,{c}^{3}+3\,{c}^{2}+15\,c+7
\right) {w}^{2}\\
&2\, \left( c+1 \right)  \left(3\,{c}^{2}+3\,c+1 \right)  \left( 3\,{c
}^{2}+6\,c+4 \right)  \left( 3\,{c}^{2}+9\,c+7 \right)  \left( {c}^{2}
+2\,c+3 \right) wu\\
&-6\, \left( {c}^{2}+c+1 \right)  \left( 3\,{c}^{2}+1
\right)  \left( 3\,{c}^{2}+9\,c+7 \right)  \left( 3\,{c}^{3}+6\,{c}^{
2}+18\,c+8 \right) wv,
\end{aligned}
\end{equation*}
\begin{equation}\label{eq:6}
\begin{aligned}
s&=- \left( c+1 \right)  \left( 3\,{c}^{2}+3\,c+1 \right)  \left( 3\,{c}^
{2}+6\,c+4 \right)  \left( 3\,{c}^{2}+9\,c+7 \right)  \left( {c}^{2}+2
\,c+3 \right) {u}^{2}\\
&+3\, \left( {c}^{2}+c+1 \right)  \left( 3\,{c}^{2
}+1 \right)  \left( 3\,{c}^{2}+9\,c+7 \right)  \left( 3\,{c}^{3}+6\,{c
}^{2}+18\,c+8 \right) {v}^{2}\\
&-3\, \left( {c}^{2}+c+1 \right)  \left( 3
\,{c}^{2}-3\,c+1 \right)  \left( 3\,{c}^{2}+6\,c+4 \right)  \left( 3\,
{c}^{3}+3\,{c}^{2}+15\,c+7 \right) {w}^{2}.
\end{aligned}
\end{equation}

See \cite{Mor} for details about finding parametric rational solutions of a homogenous polynomial of degree $2$ in several variables.

\begin{remark}\label{rem:3.2}
The above argument shows that given $p, q, r, s \in \mathbb{Q}(c, u, v, w)$, there exist $k, l, m\in\mathbb{Q}(c)$ such that the ordered pairs {\small $((c-1)^3, p), (c^3,q), ((c+1)^3,r)$}, {\small $((c+2)^3,s)$} are four rational points on \eqref{eq:3}.
\end{remark}

Next we consider the case when  $((c-2)^3, t) \in E( \mathbb{Q})$. In this case, there exists a $5-$term sequence of consecutive cubes on \eqref{eq:3}. Then one obtains 
\begin{equation}\label{eq:7}
t^2=Ku^4+Lu^3+Mu^2+Nu+P
\end{equation}
with
\begin{equation*}
K= [\left( c+1 \right) \left( 3\,{c}^{2}+3\,c+1 \right) \left( 
3\,{c}^{2}+6\,c+4 \right) \left( 3\,{c}^{2}+9\,c+7 \right)
\left( {c}^{2}+2\,c+3 \right)]^2,
\end{equation*}
\begin{equation*}
\begin{aligned}
L&=-24\, \left( c+1 \right)  \left( {c}^{2}+2\,c+3 \right)\left(3\,{
c}^{2}+3\,c+1 \right)\left( 3\,{c}^{2}+4 \right)  \left( 3\,{c}^{2}+6
\,c+4 \right) ( 3\,{c}^{2}\\
&+9\,c+7 ) ^{2} \left( {c}^{2}-c+
1 \right)  \left( 3\,{c}^{3}+3\,{c}^{2}+27\,c+1 \right) v+32c\, \left( 
c+1 \right)  \left( {c}^{2}+2\,c+3 \right)\\
&(3\,{c}^{2}+3\,c+1)\left( 3\,{c}^{2}+4 \right)\left( 3\,{c}^{2}+6\,c+4
\right)^{2} \left( 3\,{c}^{2}+9\,c+7 \right) \left( {c}^{2}+8
\right) ( 3\,{c}^{2}-6\,c\\
&+4 )w,
\end{aligned}
\end{equation*}
\begin{equation*}
\begin{aligned}
M&=6\,( 1215\,{c}^{14}+4131\,{c}^{13}+24543\,{c}^{12}+49383\,{c}^{11}+
134460\,{c}^{10}+152118\,{c}^{9}\\
&+263619\,{c}^{8}+229491\,{c}^{7}+
297153\,{c}^{6}+223859\,{c}^{5}+208374\,{c}^{4}+123018\,{c}^{3}\\
&+52116
\,{c}^{2}+16504\,c+480)  \left( 3\,{c}^{2}+9\,c+7 \right) ^{2}{v}^{2}-24\, \left( {c}^{2}+c+1 \right)  \left( 3\,{c}^{2}+4 \right)\\
&  \left( 
3\,{c}^{2}+6\,c+4 \right)  \left( 3\,{c}^{2}+9\,c+7 \right) ( 
189\,{c}^{10}+2340\,{c}^{8}-1530\,{c}^{7}+7383\,{c}^{6}\\
&-9918\,{c}^{5}+
7365\,{c}^{4}+740\,{c}^{3}-246\,{c}^{2}+1552\,c+21 ) wv+2\,( 2835\,{c}^{14}\\
&-8424\,{c}^{13}+37881\,{c}^{12}-109188
\,{c}^{11}+156276\,{c}^{10}-373032\,{c}^{9}+395718\,{c}^{8}\\
&-344976\,{c
}^{7}+374198\,{c}^{6}-336324\,{c}^{5}-98956\,{c}^{4}-407760\,{c}^{3}-
243937\,{c}^{2}\\
&+2688\,c+441 ) \left( 3\,{c}^{2}+6\,c+4 \right) 
^{2}{w}^{2},
\end{aligned}
\end{equation*}
\begin{equation*}
\begin{aligned}
N&=-72\,( {c}^{2}+c+1) ( 3\,{c}^{2}+4)( 
{c}^{2}-c+1)( 3\,{c}^{2}+1) ( 3\,{c}^{3}+6
\,{c}^{2}+18\,c+8)\\
&( 3\,{c}^{3}+3\,{c}^{2}+27\,c+1)( 3\,{c}^{2}+9\,c+7) ^{2}{v}^{3}+48\,( {c}^{2}+c+1)( 3\,{c}^{2}+4)\\
&( 3\,{c}^{2}+6\,c+4)(3\,{c}^{2}+9\,c+7)(27\,{c}^{10}+450\,{c}^{8}-450\,{c}^{7}+2019\,{c}^{6}-1494\,{c}^{5}\\
&+2325\,{c}^{4}-1180\,{c}^{3}-1398\,{c}^{2}+16\,c+21) {v}^{2}w+24\,( {c}^{2}+c+1)( 3\,{c}^{2}+4)\\
&(3\,{c}^{2}+6\,c+4)( 3\,{c}^{2}+9\,c+7)(135\,{c}^{10}+1440
\,{c}^{8}-630\,{c}^{7}+3345\,{c}^{6}-6930\,{c}^{5}\\
&+2715\,{c}^{4}+3100
\,{c}^{3}+2550\,{c}^{2}+1520\,c-21)v{w}^{2}-96c\,( {c}^{2
}+c+1)( 3\,{c}^{2}+4)\\
&( 3\,{c}^{2}-3\,c+1)( 3\,{c}^{2}-6\,c+4)( 3\,{c}^{3}+3\,{c}^{2}+15\,c+7 )( {c}^{2}+8)( 3\,{c}^{2}+6\,c+4) ^{2}{w}^{3},
\end{aligned}
\end{equation*}
\begin{equation*}
\begin{aligned}
P&=[3( {c}^{2}+c+1 )\left( 3
\,{c}^{2}-3\,c+1 \right) \left( 3\,{c}^{2}+6\,c+4 \right)
( 3\,{c}^{3}+3\,{c}^{2}+15\,c+7 )]^2{w}^{4}\\
&+72\,\left( 3\,{c}^{2}-3\,c+1 \right)  \left( 3\,{c}^{2}+4 \right) 
\left( 3\,{c}^{2}+6\,c+4 \right)  \left( 3\,{c}^{2}-9\,c+7 \right)
( 3\,{c}^{3}+3\,{c}^{2}\\
&+15\,c+7)  \left( 3\,{c}^{3}-3\,{c}^{2}+27\,c-1 \right)  \left( {c}^{2}+c+1 \right) ^{2}{w}^{3}v-18\,(6561c^{14}+2187c^{13}\\
&+91125c^{12}+42039c^{11}+287712c^{10}-5994c^9-224127c^8+6399c^7+316035c^6\\
&+232191c^5+1581642c^4+299082c^3+294228c^2+248472c-7840 ) ( {c}^{2}+c\\
&+1 ) ^{2}{w}^{2}
{v}^{2}+72\, \left( 3\,{c}^{2}+1 \right)  \left( 3\,{c}^{2}+4 \right)
\left( 3\,{c}^{2}+9\,c+7 \right) \left( 3\,{c}^{2}-9\,c+7 \right)( 3\,{c}^{3}-3\,{c}^{2}\\
&+27\,c-1 )\left( 3\,{c}^{3}+6\,{c
}^{2}+18\,c+8 \right) \left( {c}^{2}+c+1 \right)^{2}w{v}^{3}+[3\,
\left( {c}^{2}+c+1 \right)\left( 3\,{c}^{2}+1 \right)\\
&
\left( 3\,{c}^{2}+9\,c+7 \right)\left( 3\,{c}^{3}+6\,{c}^{2}+18
\,c+8 \right)]^2{v}^{4}.
\end{aligned}
\end{equation*}

We see that the above expressions are homogeneous in $v$ and $w$. So, we may assume that $w=1$. Now we consider the curve

\begin{equation}\label{eq:8}
H: Y^2=KX^4+LX^3+MX^2+NX+P
\end{equation}
over 
$\mathbb{Q}(c,v)$. This equation of the form \eqref{eq:8} is birationally equivalent to an elliptic curve $\chi$ defined by the form 
\begin{equation}\label{eq:9}
\chi: V^2=U^3-27IU-27J
\end{equation}
where
\begin{equation}\label{eq:10}
I=12KP-3LN+M^2
\end{equation}
and
\begin{equation}\label{eq:11}
J=72KMP+9LMN-27KN^2-27L^2P-2M^3.
\end{equation}
The discriminant $\Delta(\chi)$ of $\chi$ is given by $(4I^3-J^2)/27$, and the specialization of $\chi$ is singular only if $\Delta(\chi)=0$. Furthermore, the point 
\begin{equation}\label{eq:12}
R=\left( 3\dfrac{3L^2-8KM}{4K}, 27\dfrac{L^3+8K^2N-4KLM}{8K^{3/2}}\right)
\end{equation}
lies in $\chi(\mathbb{Q}(c,v))$ since $K$ is a square (see \cite[chapter 3, pp. 89-91]{JC} for details).

Now we are ready to give the following result.

\begin{theorem}\label{theo:3.3}
The curve \eqref{eq:8} is birationally equivalent over $\mathbb{Q}(c,v)$ to an elliptic curve $\chi$ with rank $\chi(\mathbb{Q}(c,v))\geqslant 1$.
\end{theorem}
\begin{proof}
Write the homogenous form of the curve \eqref{eq:8}. Then one gets $Y^2=KX^4+LX^3Z+MX^2Z^2+NXZ^3+PZ^4$ with a rational point $T=(X:Y:Z)=(1:(c+1)(3c^2+3c+1)(3c^2+6c+4)(3c^2+9c+7)(c^2+2c+3):0)$. According to the above procedure, the curve \eqref{eq:8} is birationally equivalent to \eqref{eq:9} and taking $c=3$, $v=\frac{3094}{5795}$, using \eqref{eq:10}-\eqref{eq:12}, one gets the specialization \\
\~{R} =$(\frac{4692656977319420928}{9025}, \frac{69761912906449000257785856}{9025})$ of the point $R$ on the specialized elliptic curve

\begin{equation*}
\begin{aligned}
\psi :&Y^2=X^3-\frac{19155688278708494907117216280017764352}{81450625}X\\
&+\frac{30476125037279414454071839383853830234262941440938082304}{735091890625}.
\end{aligned}
\end{equation*}
By MAGMA, \cite{BCP}, we see that the point \~{R} is a point of infinite order on $\psi$. Thus, by the Silverman specialization theorem, \cite[Theo. 11.4]{S}, the point $R$ is of infinite order on $\chi$.
\end{proof}

\begin{corol}\label{cor:3.4}
Let a nontrivial sequence of consecutive rational cubes be $(c_0-2)^3, (c_0-1)^3, c_0^3, (c_0+1)^3, (c_0+2)^3$. Then there are infinitely many elliptic curves of the form $E_j: y^2=k_jx^3+l_jx+m_j, 0\neq j \in \mathbb{Z}$, such that $(c_0+i)^3, i=-2, -1, 0, 1, 2,$ is the $x-$coordinate of a rational point on $E_j$. Furthermore, these five rational points are linearly independent.
\end{corol}
\begin{proof}
Set $c=c_0, v=v_0$ and $w=1$ in \eqref{eq:6}. Then one gets the elliptic curve
\begin{equation}\label{eq:13}
\chi_{c_0,v_0,1}: t^2=Ku^4+Lu^3+Mu^2+Nu+P
\end{equation} 
$K, L, M, N, P \in \mathbb{Q},$ and according to Theorem \ref{theo:3.3} its rank is positive. Then we will find a point $R=(u, t)$ of infinite order in $\chi_{c_0,v_0,1}(\mathbb{Q})$. Set $jR=(u_j, t_j)$ with $0 \neq j \in \mathbb{Z}$ to be the $j-th$ multiple of the point $R$ in $\chi_{c_0,v_0}(\mathbb{Q})$.

Now, substituting $c=c_0, v=v_0, w=1$ and $u=u_j$ into the formulae for $p, q, r, s \in \mathbb{Q}(c,u,v,w)$ in \eqref{eq:4}, one gets the rational numbers $p_j, q_j, r_j, s_j$ respectively.
Then substituting $p_j, q_j, r_j, s_j$ into formulae for $k, l, m \in \mathbb{Q}(c,p,q,r)$ in \eqref{eq:4} one obtains the rational numbers $k_j, l_j, m_j$ respectively.

Hence we constructed an infinite family of elliptic curves $ E_j: y^2=k_jx^3+l_jx+m_j$ with $0\neq j \in \mathbb{Z}$. This infinite family $E_j$ of elliptic curves has the point  $((c_0-1)^3, p_j), (c_0^3, q_j), ((c_0+1)^3, r_j), ((c_0+2)^3, s_j), ((c_0-2)^3, t_j) \in E_j(\mathbb{Q})$. This means that we get an infinite family of elliptic curves with a $5-$term sequence of rational points whose $x-$coordinates consist of a sequence of consecutive cubes in $ \mathbb{Q}$.

Now, we will show that the points  $((c_0-1)^3, p_j), (c_0^3, q_j), ((c_0+1)^3, r_j), ((c_0+2)^3, s_j), ((c_0-2)^3, t_j) \in E_j(\mathbb{Q})$ are linearly independent. To do this we first need to find a point $(u,t)$ lying in \eqref{eq:13}.

Consider the equation \eqref{eq:13}. Taking $c=3, v=\frac{3094}{5795}, w=1$, we obtain the curve

\begin{equation}\label{eq:14}
\begin{aligned}
{t}^{2}&=63404527588416\,{u}^{4}-{\frac {782109496219903488}{9025}}\,{u}^{2}+{\frac{19793578415844699648}{550525}}\,u\\
&+{\frac {478172417894196583574016}{303077775625}}.
\end{aligned}
\end{equation}
Completing the square on the RHS of \eqref{eq:14}, we obtain a point of infinite order
\begin{equation*}
(u,t)=(\frac{60547}{77653}, \frac{78134116669224}{130068775})\in \chi_{3,\frac{3094}{5795},1}(\mathbb{Q}).
\end{equation*}
So, using the specialization $c=3$, $v=\frac{3094}{5795}$, $w=1$, $u=\frac{60547}{77653}$ we obtain a specialized elliptic curve

\begin{equation*}
\begin{aligned}
E:y^2=&\frac{1019317647604532728704}{50501152925375}x^3+\frac{170640863010859366860672}{2657955417125}x+\\
&\frac{5018469623203840351296469056}{16917886230000625}
\end{aligned}
\end{equation*}
 with the following set of consecutive cubes in
 \begin{equation*}
 \begin{aligned}
 E(\mathbb{Q})=&(1,\frac{78134116669224}{130068775}), (2^3, \frac{117823324221624}{130068775}),(3^3, \frac{202645347682344}{130068775}),\\
&(4^3, \frac{405025200935544}{130068775}), (5^3, \frac{898732973533416}{130068775}).
 \end{aligned}
\end{equation*} 
 By MAGMA, we see that these rational points are linearly independent.
 
 By the Silverman Specialization Theorem, we can say that the points\\
$((c-1)^3, p_j), (c^3, q_j), ((c+1)^3, r_j), ((c+2)^3, s_j), ((c-2)^3, t_j)$ are linearly independent in $E_j$ over $\mathbb{Q}(c,v,u_j)$. This completes the proof of the corollary.
\end{proof}

\begin{remark}\label{rem:3.5}
The previous corollary implies the existence of an infinite family of elliptic curves whose rank $r \geq 5$.
\end{remark}

Finally, if we want to construct a $6-$term sequence of consecutive cubes, we assume that the point $((c+3)^3,z)$ is a rational point on \eqref{eq:3}. So, the following relation is satisfied
\begin{equation*}
z^2=K'u^4+L'u^3+M'u^2+N'u+P'
\end{equation*} with $K', L', M', N', P' \in \mathbb{Q}(c,v)$. So we give the following remark.

\begin{remark}\label{rem:3.6}
The existence of a $6-$term sequence of consecutive cubes on the elliptic curve depends on the existence of a rational point $(u, t, z)$ on the algebraic curve defined by the following intersection
\begin{equation*}
C: t^2=Ku^4+Lu^3+Mu^2+Nu+P, \ z^2=K'u^4+L'u^3+M'u^2+N'u+P'.
\end{equation*}
The genus of the curve $C$ is $5$. Thus , Falting's Theorem says that for given $c\in \mathbb{Q} $, there are only finitely many elliptic curves over $\mathbb{Q}$ defined by $y^2=kx^3+lx+m$ where $(c+j)^3, j=-2,-1,0,1,2,3$ are the $x-$coordinates of a $6-$term sequence of consecutive cubes.
\end{remark}

%################################%
\subsection*{Acknowledgements}

We would like to thank the referee for carefully reading our manuscript and for giving such constructive comments which substantially helped improving the presentation of the paper. And also we thank Professor Mohammad Sadek for his useful discussions. The second author was supported by the Research Fund of Uluda\u{g} University under Project No: F-2016/9.

\end{document}